\documentclass[reqno,a4paper]{amsart}
\usepackage[utf8]{inputenc}
\usepackage[T1]{fontenc}
\usepackage[british]{babel}
\usepackage{amsmath,amsthm,amssymb}
\usepackage[foot]{amsaddr}

\usepackage{mathrsfs}

\usepackage{graphicx}
\usepackage{wrapfig}
\usepackage{subfig}
\usepackage{enumitem}
\usepackage[usenames,dvipsnames]{color}
\usepackage{dsfont}
\usepackage{microtype}

\usepackage{nomencl}

\usepackage{etoolbox}

\usepackage{mathtools}
\mathtoolsset{showonlyrefs,showmanualtags}

\usepackage{pstricks,pst-plot,pst-math} 
\usepackage{pstricks-add} 

\usepackage{esint}

\newtheorem{thm}{Theorem}[section]

\newtheorem{lem}[thm]{Lemma}
\newtheorem{prop}[thm]{Proposition}

\theoremstyle{definition}

\theoremstyle{remark}
\newtheorem{rem}[thm]{Remark}

\numberwithin{equation}{section}

\newcommand{\DeclareAutoPairedDelimiter}[3]{%
	\expandafter\DeclarePairedDelimiter\csname Auto\string#1\endcsname{#2}{#3}%
	\begingroup\edef\x{\endgroup
		\noexpand\DeclareRobustCommand{\noexpand#1}{%
			\expandafter\noexpand\csname Auto\string#1\endcsname*}}%
	\x}

\DeclareAutoPairedDelimiter{\abs}{\lvert}{\rvert}
\DeclareAutoPairedDelimiter{\norm}{\lVert}{\rVert}
\DeclareAutoPairedDelimiter{\bra}{(}{ )}
\DeclareAutoPairedDelimiter{\pra}{[}{]}
\DeclareAutoPairedDelimiter{\set}{\{}{\}}
\DeclareAutoPairedDelimiter{\skp}{\langle}{\rangle}

\DeclareMathAlphabet{\mathup}{OT1}{\familydefault}{m}{n}
\newcommand{\dx}[1]{\mathop{}\!\mathup{d} #1}

\DeclareMathOperator*{\supp}{supp}

\newcommand{\N}{\mathds{N}}
\newcommand{\R}{\mathds{R}}

\usepackage[colorinlistoftodos,prependcaption,textsize=tiny]{todonotes}

\newtoggle{Nebenrechnungen}
\togglefalse{Nebenrechnungen}

\newtoggle{draft}
\togglefalse{draft}

\usepackage{hyperref}
\definecolor{darkblue}{rgb}{0,0,0.6}
\hypersetup{
    pdftitle={Front blocking in the presence of gradient drift},    
    pdfauthor={Simon Eberle},
    colorlinks=true,       
    linkcolor=darkblue,          
    citecolor=darkblue,        
    filecolor=darkblue,      
    urlcolor=darkblue           
}

\author{Simon Eberle$^1$}
\address{$^1$Faculty for mathematics, University of Duisburg-Essen.}

\email{simon.eberle@uni-due.de}

\title[]{Front blocking in the presence of gradient drift}

\let\rho\varrho
\let\epsilon\varepsilon

\begin{document}

\begin{abstract}
\noindent
In this paper we derive quantitative conditions under which a compactly supported drift term blocks the propagation of a traveling wave in a straight cylinder in dimension $n \geq 3$ under the condition that the drift has a potential.

\end{abstract}
	\maketitle

\section{Introduction}

This paper is an extension of paper \cite{front_blocking_vs_propagation_drift} where the author has given an explicit a-priori condition for blocking of traveling waves in a cylindrical domain subject to (compactly supported) drift to occur in \emph{ one spatial dimension}. In this article we are going to discuss the case $n \geq 3$. The main assumption we make on the drift term is that has a potential.

To the best knowledge of the author there are no quantitative results on that matter available apart from \cite{front_blocking_vs_propagation_drift}. So we hope to contribute to the understanding of blocking of traveling waves.

The investigation of traveling waves in cylinders, also subject to drift, has been done in depth in the seminal paper \cite{traveling_fronts_in_cylinders}. However the drift term has been required to be independent of the direction of propagation, in order to allow for classical traveling waves. Since then the notion of traveling waves has been broadened to more general media, i.e. pulsating fronts for periodic media \cite{BerestyckiHamelPeriodicExcitableMedia} and the very general transition fronts for very general media \cite{BerestyckiHamelGeneralizedTransitionFronts}.
In recent years there have been investigations of existence and non existence of transition fronts in outer domains with a compactly supported obstacle \cite{MatanoObst}, in cylinders with varying nonlinearity \cite{Zlatos,change_of_speed_1} and, with respect to this work especially interesting, in opening or closing cylinders \cite{front_blocking,suden_opening,change_of_speed2}.

The subject of this paper are entire solutions of the generalized initial value problem

	\begin{align} \label{DiffEqu}
	\begin{cases}
	\partial_t u - \Delta u +k(x) \cdot \nabla u = f(u) &\text{ in } D, \\
	\frac{\partial u}{\partial \nu} = 0 &\text{ on } \partial D, \\
	u(t,x) - \phi(x_1+ct) \rightarrow 0 \text{ as } t \rightarrow -\infty &\text{ uniformly in } D,
	\end{cases}
	\end{align}
	where $D:= \R \times \Omega$, $\Omega \subset \R^{n-1}$ a smooth domain, $k \in C_c(\bar{D},\R^ n)$, $\supp k \subset [-x_0,0] \times \bar{\Omega}$, $x_0>0$ and $f$ is a bistable nonlinearity (for details see section \ref{section:notation_assumption}). 
Furthermore the major assumption will be that $k$ has a \emph{potential}, i.e. there is $H \in C^1(D)$ s.t. 
\begin{align}
	\nabla H = k \quad \text{ in } D.
\end{align}

	\iftoggle{draft}{%
}
{
	\begin{figure}[!h]
		\psset{xunit=1cm,yunit=1cm}
		\frame{
			\begin{pspicture*}
			(-5.2,-1.2)(7.1,2.3)
			\psframe[fillstyle=vlines,hatchsep=0.2,hatchangle=120](2,2)
			\psline(-5,2)(7,2)
			\psline(-5,0)(7,0)
			\psline(0,0)(0,-0)
			\psline(2,0)(2,-0)
			\rput[lb](0.5,-0.5){ \boldmath$ k(x)$}
			\rput[lb](-0.4,-0.8){$-x_0$}
			\rput[lb](1.9,-0.8){$0$}
			\end{pspicture*}
		}
		\caption{Infinite cylinder with transition zone}
	\end{figure}

}



We are able to give an explicit criterion for blocking involving the net drift and some term that takes into account the concentrations of $k$ as formulated in the following theorem.
\begin{thm} \label{theorem:blocking}
Let $n \geq 3$. Then there is a constant $C(f,n)>0$ only depending on $f$ and $n$ and constants $C_i(\partial \Omega,f,n)$ only depending on the Lipschitz norm of $\partial \Omega$, $f$ and $n$ (all of them explicit) such that if  
	\begin{align} \label{blocking_condition}
	\begin{split}
	C_1(n,f)  &>\exp \bra { - \int \limits_{-x_0}^{0} \fint \limits_\Omega k_1(x_1,y) \dx{x_1} \dx{y}       }
	\max \left \{ C_2(n,f,\partial \Omega),  C_3(n,f,\partial \Omega) \right . \\
& \left .  \norm { \exp \bra { - \int \limits_{-x_0}^{x_1} k_1(\zeta,y)  \dx{\zeta}    }  }^{n-1}_{L^\infty(D_{-x_0}^0)}
	  \bra {  \int \limits_{D_{-x_0}^0} \exp \bra { \int \limits_{-x_0}^{x_1} k_1(\zeta,y) \dx{\zeta}      }^{n+1} \dx{x_1} \dx{y}    }^\frac{n}{n+1}           \right    \} 
	  \end{split}
	\end{align}
	holds, the unique solution of \eqref{DiffEqu} is blocked to the left, i.e. there exists a stationary supersolution $w: D \rightarrow \R$ of \eqref{DiffEqu} such that
	\begin{align}
	u(t,x) \leq w(x) \quad \text{ for all } t \in \R, x \in D 
	\end{align}
	and $w(x) \rightarrow 0$ as $x_1 \rightarrow -\infty$.
\end{thm}

The strategy of the proof of Theorem  \ref{theorem:blocking} is to construct the stationary supersolution $w$ as local minimizer of an appropriate functional in some weighted Sobolev space. The main observation is that \eqref{DiffEqu} becomes variational if the drift term is encoded in some weight. With this `trick' we are able to use ideas from \cite{front_blocking}, where the authors show that a neck can be introduced into a given tube in such a way that propagation gets blocked by constructing a stationary supersolution that vanishes behind the neck.

The paper is organized as follows. First we clarify assumptions and notation. Then, for the sake of completeness we shortly address the question of existence and uniqueness in section \ref{section:existence_and_uniqueness}. In section \ref{section:blocking} we give the strategy of the proof and the proof of Theorem \ref{theorem:blocking}.

\begin{rem}
	The case $n=2$ can be handled by extending the solution to three dimensions, i.e.
	\begin{align}
		&\tilde{u}(x_1,x_2,x_3) := u(x_1,x_2) ~,~ \tilde{k}(x_1,x_2,x_3) :=(k_1(x_1,x_2), k_2(x_1,x_2),0) \\ &\text{ and } \tilde{H}(x_1,x_2,x_3) := H(x_1,x_2)
	\end{align}
	and assuming without loss of generality that $\Omega = (0,r)$, by setting (e.g.) $\tilde{\Omega} =(0,r) \times (0,1)$.
	Then Theorem \ref{theorem:blocking} provides a criterion for blocking that can be projected back onto $n=2$.
\end{rem}

\begin{rem}
	An analogous criterion for almost unchanged propagation as in the one-dimensional case in \cite{front_blocking_vs_propagation_drift} does also hold in the $n$-dimensional case, ie.
	\begin{thm} [almost unchanged propagation]
		There is a constant $C(f,x_0)>0$  (only depending on $f$ and $x_0$ which is explicitly given) such that if
		$K:= \max \set { \max_{x \in D} k_1(x),0}$ is small enough to satisfy 
		\begin{align} \label{condition_theorem_propagation}
		K  \leq C(f,x_0),
		\end{align}
		then the unique solution $u$ of \eqref{DiffEqu} converges to a traveling wave with profile $\phi$ and speed $c$, i.e.
		\begin{align}
		u(t,x) - \phi(x_1+ct +\beta) \rightarrow 0 \text{ as } t \rightarrow +\infty \text{ uniformly in } D, 
		\end{align}
		where $\beta \in \R$ is a constant shift.
	\end{thm} 
\end{rem}

\begin{rem} 
	For $n=1$ even though we did only derive the criterion for $n \geq 3$ the expression behaves as in \cite{front_blocking_vs_propagation_drift}. But due to the Sobolev-embeddings becoming weaker in increasing dimensions, this generalized criterion is weaker,
in the sense that the criterion for $n>m$ might not be satisfied for an $m$-dimensional drift term constantly extended to $n$ dimensions, that satisfies the criterion in $m$ dimensions.
\end{rem}

\begin{rem}
	It has been proven in \cite{suden_opening} that in three dimensions abrupt opening of a channel that is also steep enough leads to blocking of a generalized solution of \eqref{DiffEqu}. This suggests that a very concentrated drift term with sufficiently big integral should satisfy criterion \eqref{blocking_condition}. Indeed, if we set
	\begin{align}
		k^{\epsilon,C}(x) := \frac{C}{\epsilon} \chi_{[-\epsilon,0]}(x_1) e_1
	\end{align}
	for $\epsilon, C >0$. Then criterion \eqref{blocking_condition} holds if $C$ is big enough such that
	\begin{align}
		\frac{C_1(n,f)}{C_3(n,f,\partial \Omega)} > \exp \bra { - \int \limits_{-x_0}^0 \fint \limits_\Omega k_1^{\epsilon,C} } = \exp(-C)
	\end{align}
	and $\epsilon$ is small enough such that
	\begin{align}
		\frac{C_2(n,f,\partial \Omega)}{C_3(n,f,\partial \Omega)} &>   \bra {  \int \limits_{D_{-x_0}^0} \exp \bra { \int \limits_{-x_0}^{x_1} k_1^{\epsilon,C}(\zeta,y) \dx{\zeta}      }^{n+1} \dx{x_1} \dx{y}    }^\frac{n}{n+1}  \\
		&= \bra { \frac{\abs {\Omega} }{(1+n) C} \bra { \exp((1+n)C) -1} }^\frac{n}{n+1} \epsilon^\frac{n}{n+1} .
	\end{align} 
	(Note that here $x_0(\epsilon)=\epsilon$.)
\end{rem}

\section*{Acknowledgement}
We thank Prof. Dr. G.S. Weiss for fruitful discussions.

\section{Notation and assumptions} \label{section:notation_assumption}
In this section we clarify the assumptions we make and give the notation that shall be used in the following.
Let $D$ be the cylindrical domain as specified above. The nonlinearity $f$ shall be of bistable type and obey
\begin{align}
\label{F_1} &f \in C^{2}([0,1]), \\
\label{F_2} &f(0)  = 0, f(1) = 0, \\
\label{F_3}&f'(0) <0 , f'(1) <0, \\
\label{F_4}  &f<0 \text{ on } (0, \theta), f >0 \text{ on } (\theta, 1) \text{ for some } \theta \in (0,1), 
\\
\label{F_8}& \int \limits_0^1 f(u) \dx{u}  >0 .
\end{align}

Then there is a unique speed $c>0$ and unique (up to translation) traveling wave profile $\phi$ for the nonlinearity $f$, i.e
there are speeds $c >0$ and wave profile $\phi: \R \rightarrow \R$, such that (see e.g. \cite{FifeMcLeod}) 
\begin{align} \label{ODE_travelling_wave}
\begin{cases}
\phi^{\prime \prime}(z) - c \phi^\prime(z) + f(\phi(z)) =0 \text { for all } z \in \R, \\
\phi(- \infty) = 0, \phi(+ \infty)=1, \\
0<\phi(z) < 1\text{ for all } z \in \R ,\\
\phi'(z) >0 \text{ for all } z \in \R.
\end{cases}
\end{align}

\section{Existence and uniqueness} \label{section:existence_and_uniqueness}

For the sake of completeness and to ensure the reader that we are not investigating the empty set of solutions or assume uniqueness of solutions of \eqref{DiffEqu} without justification, let us mention that existence and uniqueness  for solutions of \eqref{DiffEqu} can be obtained almost literally copying the proof of Theorem 2.1 in \cite{MatanoObst} or Appendix A in \cite{change_of_speed_1}.

\section{A necessary condition for propagation / a sufficient condition for blocking} \label{section:blocking}

The objective of this section shall be the proof of Theorem \ref{theorem:blocking}.
The proof of our result on blocking will rely mainly on the observation that the stationary version of problem \eqref{DiffEqu} 
\begin{align}
	-\Delta w +k \cdot \nabla w =f(w) \text{ in } A \subset D
\end{align}
is variational with functional
\begin{align}
J_A(w) = \int \limits_A \bra { \frac{1}{2} \abs {\nabla w }^2 +F(w)     } \psi(x) \dx{x},
\end{align}
where $F(t):= \int_t^1 f(s) \dx{s}$ and $\psi$ is given by
\begin{align} \label{Def:psi}
\psi(x) := \exp (-H(x))  \quad \text{ in } D.
\end{align}
Therefore $\psi(x)>0$ for all $x \in D$ and $\psi$  is qualified as weight function. With this trick of encoding the drift term in a weight function we are now in the position to use variational techniques to construct a local minimizer of the functional $J$ that will then be extended to a stationary supersolution.
The strategy of this proof is inspired by the strategy used in \cite{front_blocking} where the authors show that a thin neck can be introduced into a given channel in such a way that a traveling wave gets blocked.

Let us briefly describe our strategy in the following. The goal is to construct $w$ such that
\begin{align}
	u(t,x) \leq w(x) \quad \text{ for all } x \in D, t \in \R
\end{align}
and $w(x)\rightarrow 0$ as $x_1 \rightarrow -\infty$, which will be possible if condition \eqref{blocking_condition} is met.

To make $J$ well defined and to ensure that $F$ grows quadratically at infinity $f$ shall be extended linearly to a function $f \in C^{1,1}(\R)$.
Let us introduce the following shorthand notation
\begin{align}
	D_a^b &:= \set {  a<x_1<b} \cap D = (a,b) \times \Omega \text{ for } a<b \text{ and } \\
	H^1(A,\psi) &:= H^1(A,\psi \dx{x}).
\end{align}
In order to construct such a supersolution we

\begin{enumerate}
	\item first show that for any $R<-x_0-1$ (arbitrary but fixed) there is $\delta(f, \psi,k, a)>0$ independent of $R$ such that 
	\begin{align}
		J_{D_R^a} (w) > J_{D_R^a} (w_0)
	\end{align} 
	for all $w \in H^1_{0,1} (D_R^a, \psi )$ such that $\norm {w-w_0}_{H^1(D_R^a, \psi )} = \delta$ and $w_0(x) = w_0(x_1) := \frac{x_1}{a} \chi_{[0,a]}(x_1)$, $a>0$ is an auxiliary constant that can be chosen in an optimal way (depending on $f$) and we understand
	\begin{align}
		H^1_{0,1}(D_c^d, \psi) := \set { v \in H^1(D_c^d, \psi) : v(c,y)=0, v(d,y)=1 \text{ for almost all } y \in \Omega}
	\end{align}
	(where boundary values are understood in the sense of traces).
	From this we can conclude by the direct method, that there is a local minimizer $w_R \in H^1_{0,1} (D_R^a, \psi )\cap \set { \norm {w-w_0}_{H^1(D_R^a, \psi)} \leq \delta }$ that is a weak solution of
	\begin{align}
		-\Delta w_R - \frac{\nabla \psi}{\psi} \cdot \nabla w_R = f(w_R) &\text{ in } D_R^a, \\
		\frac{\partial w_R}{\partial \nu} =0 &\text{ on } (R,a) \times \partial \Omega , \\
		w_R(R,y)=0, w_R(a,y)=1 &\text{ for almost all } y \in \Omega
	\end{align}	
	with $0<w_R<1$ in $D_R^a$ (by comparison principle).
	\iftoggle{Nebenrechnungen}{%
		\\
		\textcolor{blue}{ 
			\hrulefill \\
			\emph{Nebenrechnung:}
	Wir können in die schwache Formulierung (mithilfe von Approximation) $H^1$-Funktionen mit Nullrandwerten einsetzen. D.h. wir wissen
	\begin{align}
		\int \bra { \nabla w_R \cdot \nabla \phi -f(w_R) \phi      } \psi \dx{x} =0 \quad \forall \phi \in \mathring{H}^1
	\end{align}
	Damit ist $w_R^- := \min \set {w_R,0}$ eine zulässige Testfunktion und wir bekommen 
\begin{align}
0 \leq	\norm {\nabla w_R^-}^2_{L^2_\psi} = \int f(w_R) w_R^- \underbrace{\psi}_{>0}
\end{align}
Ferner wissen wir, dass, wenn $w_R^- \neq 0$, dann ist $w_R <0$ und damit $f(w_R)>0$. Also kann die Ungleichung nur gelten, wenn $w_R^- \equiv 0$ fast überall. Damit folgt $0 \leq w_R$ fast überall. 
Man kann mit $(w_R-1)^+$ testen (es ist eine zulässige Testfunktion).
Da $k$ und damit $\psi$ glatt sind und $f \in W^{2,\infty}$ folgt mit der einfachen Regularitätstheorie aus Evans (§6.3 Thm 2 und General Sobolev Inequality §5.6 Thm 6) schon, dass die Lösung $w_R \in C^{2,\gamma}$ und damit liefert das starke Maxiumsprinzip für klassische Lösungen (Evans §6.4 Thm 3) sofort (,da $w_R$ aufgrund der Randwerte nicht Konstant sein kann), dass $w_R$ sein Maximum und Minimum, also $0$ und $1$ nicht im Inneren annehmen kann.	
	\hrulefill 
\\	
}
}{%
}
\item In a next step we pass to the limit $R \rightarrow -\infty$ exploiting that $\delta$ is independent of $R$ and show that the limit $w_\infty$ solves
\begin{align}
	-\Delta w_\infty - \frac{\nabla \psi}{\psi} \cdot \nabla w_\infty = f(w_\infty) &\text{ in } D_{-\infty}^a, \\
\frac{\partial w_\infty}{\partial \nu} =0 &\text{ on } (-\infty,a) \times \partial \Omega , \\
w_\infty(-\infty,y)=0, w_\infty(a,y)=1 &\text{ for almost all } y \in \Omega 
\end{align}
and it follows for such a solution (by the strong maximum principle) that $0<w_\infty<1$ in $(-\infty,a) \times \Omega$.
\item In the last step we show that if we extend $w_\infty$ by $1$ into $[a,\infty) \times \Omega$ it is a supersolution of \eqref{DiffEqu}.
\end{enumerate}

\begin{figure}[!h]
	\psset{xunit=1cm,yunit=1cm}
	\frame{
		\begin{pspicture*}
		(-5.2,-1.2)(7.1,2.3)
		\psframe[fillstyle=vlines,hatchsep=0.2,hatchangle=120](2,2)
		\psline(-5,2)(7,2)
		\psline(-5,0)(7,0)
		\psline(0,0)(0,-0)
		\psline(2,0)(2,-0)
		\rput[lb](0.5,-0.5){ \boldmath$ k(x)$}
		\rput[lb](-0.4,-0.8){$-x_0$}
		\rput[lb](1.9,-0.8){$0$}
		\rput[lb](3.9,-0.8){$a$}
		\rput[lb](-4,-0.8){$R$}
		\psline[linestyle=dashed](3.9,0)(3.9,2)
		\psline[linestyle=dashed](-4,0)(-4,2)
		\rput[lb](-2,-0.8){$-x_0-1$}
		\psline[linestyle=dashed](-1,0)(-1,2)
		\end{pspicture*}
	}
	\caption{Infinite cylinder with transition zone}
\end{figure}

\begin{prop} \label{proposition:blocking:fixed_R}
	Assume that condition \eqref{blocking_condition} holds, then for all $R<-x_0-1$ there is 
	\begin{align} \label{Def:delta} 
	\delta &:= \bra {\frac{\alpha}{2}}^\frac{1}{q-2} \psi^\frac{1}{2}(-x_0) \max \left \{   \tilde{\gamma} \bra { \frac{\alpha}{4},q   } C_1^q(\partial \Omega)  , \right . \\
	& \left . \left ( \tilde{\gamma} \bra {\frac{\alpha}{2} , m} C_2^m(\partial \Omega)  2^\frac{(2-q)m}{2p} \norm {  \exp \bra {-  \int \limits_{-x_0}^{x_1} k_1(\zeta,y) \dx{\zeta} }   }_{L^\infty(D_{-x_0}^0)}   \right . \right . \\
	& \left . \left .\pra {\int \limits_{D_{-x_0}^0}     \exp \bra {\frac{j}{2}   \int \limits_{-x_0}^{x_1} k_1(\zeta,y) \dx{\zeta} } \dx{x_1} \dx{y}   }^\frac{m}{j}  \right )^\frac{q-2}{m-2}  \right \}^{-\frac{1}{q-2}}   >0
	\end{align}
	(independent of $R$)
	such that for any $R<-x_0-1$ there is a local minimizer $w_R \in H^1_{0,1}(D_R^a, \psi ) \cap \set { \norm {w_R - w_0}_{H^1_{(D_R^a, \psi )}}    < \delta  }$ of 
	\begin{align}
	J_{D_R^a} 
	\end{align}
	in $H^1_{0,1}(D_R^a, \psi ) \cap \set { \norm {w_R - w_0}_{H^1_{(D_R^a, \psi )}}    \leq \delta  }$.
	The constant
	\begin{align} \label{Def:alpha}
		\alpha:= \min \set { \frac{1}{4} , -\frac{f'(0)}{4}   } >0
	\end{align}
	only depends on $f$, $\tilde{\gamma}$ as defined in \eqref{def:gamma_tilde} depends on $n$ and $f$ and the other constants depend on $n$ and are given by $q:= 2^* = \frac{2n}{n-2},p = 2 \frac{n+1}{n+2},m:=p^*= \frac{2n}{n-1}
	, j:=2(n+1)$. 
	
\end{prop}

In order to prove this we will split up $D_R^a$ into the part $D_0^a$ where $\psi$ is constant and $w_0$ is linear and $D_R^0$ where $w_0 \equiv 0$ and $\psi$ does encode the behaviour of $k$.

On the second subset we will exploit the following Lemma.

\begin{lem} \label{lemma:auxiliary_lemma_fixed_R}
	With $\delta$ and $\alpha$ given as in \eqref{Def:delta}, \eqref{Def:alpha} in Proposition \ref{proposition:blocking:fixed_R}, it holds that for all $w \in H^1_0(D_R^0,\psi ) \cap \set { \norm {w-w_0}_{H^1(D_R^0, \psi ) } = \norm {w}_{H^1(D_R^0, \psi )} \leq \delta  }$
	\begin{align}
	J_{D_R^0}(w) \geq J_{D_R^0} (w_0) + \alpha \norm { w  }^2_{H^1(D_R^0, \psi ))},
	\end{align}
	where
	$H^1_0(D_R^ 0, \psi )) : = \set { w \in  H^1(D_R^ 0, \psi )) : w(R,y)=0 \text{ for almost all } y \in \Omega } $.
\end{lem}

\begin{proof}[Proof of the Lemma]
	First by a Taylor expansion of $F$ we find that 
	
	\begin{align} \label{Taylor_expansion_of_functional}
	J_{D_R^0}(w) = \int \limits_{D_R^ 0} \bra {  \frac{1}{2}  \abs {\nabla w}^2 +F(0) +F'(0) w + \frac{1}{2} F''(0) w^2 + \eta(w) w^2     }  \psi \dx{x} .
	\end{align}
	We rewrite this as 
	\begin{align}
	J_{D_R^ 0}(w) = J_{D_R^0} (w_0) + \int \limits_{D_R^ 0} \bra { \frac{1}{2} \abs {\nabla w}^2 - \frac{1}{2} f'(0) w^2 + \eta(w) w^2    } \psi .
	\end{align}
	It is immediate that
	\begin{align}
	\int \limits_{D_R^0} \bra { \frac{1}{2} \abs {\nabla w}^2 -      \frac{1}{2} f'(0) w^2    } \psi \geq  \underbrace{  \min \set { \frac{1}{2} , -\frac{f'(0)}{2}   } }_{=:2\alpha}  \norm {w}^2_{H^1(D_R^0, \psi )} .
	\end{align}
	It remains to absorb the last term in the Taylor expansion into this. 
	In order to do so we will use that 
	\begin{align}
	\abs { \int \limits_{D_R^0} \eta(w) w^ 2    } \leq \sigma \bra {\psi,k,x_0 , \norm { w  }_{H^1(D_R^ 0, \psi )}  } \norm { w  }^2_{H^1(D_R^ 0, \psi )} ,
	\end{align}
	where $\sigma$ is independent of $R$ and $\sigma \rightarrow 0$ as $ \norm { w  }_{H^1(D_R^0, \psi )} \rightarrow 0$.

	First of all we estimate the error-term in the Taylor-expansion $\eta(w) w^2$. By definition
	\begin{align}
		\eta(s) s^2 = F(s) - \bra { F(0) + F'(0)s + \frac{1}{2} F''(0) s^ 2    } .
	\end{align}
	
	\begin{itemize}
		\item for $s \in (-\infty,0]$:
		\begin{align}
			\eta(s) s^2 &= \int \limits_s^1 f(t) \dx{t} - \bra { F(0) +F'(0)s + \frac{1}{2} F''(0) s^2     }  \\
			&= \underbrace { \int \limits_0^1 f(t) \dx{t} }_{=F(0)} + \int \limits_s^ 0 f'(0) t \dx{t}  - \bra { F(0) - \frac{1}{2} f'(0) s^2      } =0
		\end{align}
		
		\item for $s\in [0,1]$:
		Since $f \in C^2([0,1])$ hence $F \in C^3([0,1])$ from Taylor's Theorem we know that
		\begin{align}
			\eta(s) \leq \norm {f''}_{L^\infty} s \text{ hence in this regime } \eta(s) s^2 \leq \norm {f''}_{L^\infty} s^3 . 
		\end{align}
		
		\item for $s \in [1,\infty)$ we have
		\begin{align}
			\eta(s) s^2 &= \int \limits_s^1 f'(1) (t-1) \dx{t} - \bra { F(0) - \frac{1}{2} f'(0) s^2   } \\
			&= -\frac{1}{2} f'(1) (s-1)^2 -F(0) + \frac{1}{2} f'(0) s^2.
		\end{align}
		This implies that
		\begin{align}
			\abs {\eta(s) s^ 2} \leq  \underbrace { \bra {  \underbrace {  F(0) }_{\geq 0} \underbrace {- \frac{1}{2} \bra {f'(0) +f'(1) }}_{\geq 0}      }  }_{=: \mu(f)}  s^2   
		\end{align}
	\end{itemize}
	Putting everything together we find that
	\begin{align}
		\abs {\eta(s) s^2} \leq \norm {f''}_{L^ \infty} s^3 \chi_{(0,1)}(s) +\mu(f) s^2 \chi_{[1,\infty)}(s) .
	\end{align}
	
	We claim furthermore that for any $\gamma >0$ and $q>2$ there is $\tilde{\gamma}(\gamma,q,f,n) \geq 0$ s.t. 
	\begin{align}
		\abs {\eta(s) s^ 2 } \leq \gamma s^2 + \tilde{\gamma} \abs {s}^q \quad \text{ for all } s \in \R. 
	\end{align}
	
	\begin{itemize}
		\item $s \in (-\infty,0]$: nothing is to do.
		\item $s \in (0,1)$:
		\begin{itemize}
			\item if furthermore $\norm {f''}_{L^ \infty} s^3 \leq \gamma s^ 2$ nothing is to do.
			\item if otherwise $\norm {f''}_{L^ \infty} s^3 > \gamma s^ 2$ then
			\begin{align}
				\abs { \eta(s) s^2   } &\leq \norm {f''}_{L^ \infty} s^3 = \norm {f''}_{L^ \infty} s^ {3-q} s^q \\ &\leq 
				\begin{cases}
				\norm {f''}_{L^ \infty} s^q &\text{ if } 3-q \geq 0 \\
				\norm {f''}_{L^ \infty} \bra {\frac{\gamma}{\norm {f''}_{L^ \infty}}}^{3-q} s^q  &\text{ if } 3-q < 0
				\end{cases}
			\end{align}
		\end{itemize}
		\item $s\in [1,\infty)$: Since $q>2$ choosing $\tilde{\gamma} := \mu(f)$ does it in this regime.
	\end{itemize}
	To sum it all up we set
	\begin{align} \label{def:gamma_tilde}
		\tilde{\gamma}(\gamma,q,f,n) := \begin{cases}
	\max \set {	\norm {f''}_{L^ \infty}, \mu(f)} &\text{ if } 3-q \geq 0, \\
	\max \set {	\norm {f''}_{L^ \infty} \bra {\frac{\gamma}{\norm {f''}_{L^ \infty}}}^{3-q}, \mu(f) } &\text{ if } 3-q < 0 .
		\end{cases}
	\end{align}
	(In the following we are going to suppress the dependence on $f$ and $n$.)
	We know that in the case of $q := 2^* = \frac{2n}{n-2} >2$ the Sobolev-embedding 
	\begin{align}
		H^1(D) \hookrightarrow L^q(D), \quad \norm {w}_{L^q(D)} \leq C(\operatorname{Lip}(\partial D)) \norm {w}_{H^1(D)} 
	\end{align}
	does only depend on the Lipschitz-norm of the boundary of $D$ but is independent of the measure of the set $D$, we have that
	\begin{align}
		\norm {w}_{L^q(D_R^{-x_0})} \leq C_1(\operatorname{Lip}(\partial \Omega)) \norm {w}_{H^1(D_R^{-x_0})} ,
	\end{align}
	since $D$ is a cylindrical domain.
	
	Using this Sobolev-embedding we calculate (exploiting that $\psi$ is constant in $D_R^{-x_0}$):
	\begin{align}
		\int \limits_{D_R^{-x_0}} \abs { \eta(w) w^2   } \psi &\leq \int \limits_{D_R^{-x_0}} \bra { \gamma w^2 + \tilde{\gamma}(\gamma,q)  \abs {w }^q } \psi \\
		&\leq \gamma \norm {w}^2_{L^2(D_R^{-x_0}, \psi )} + \tilde{\gamma}(\gamma,q) \psi(-x_0) \norm {w}^q_{L^q(D_R^{-x_0})} \\
		&\leq \gamma \norm {w}^2_{L^2(D_R^{-x_0}, \psi )} + \tilde{\gamma}(\gamma,q) \psi(-x_0) C_1^q(\operatorname{Lip}(\partial \Omega)) \norm {w}^q_{H^1(D_R^{-x_0})}  \\
	&\leq \gamma \norm {w}^2_{L^2(D_R^{-x_0}, \psi )} + \tilde{\gamma}(\gamma,q) \psi^{1-\frac{q}{2}} (-x_0) C_1^q(\operatorname{Lip}(\partial \Omega)) \norm {w}^q_{H^1(D_R^{-x_0},\psi )} 
		\end{align}
	For the part where $\psi$ is not constant, i.e. in $D_{-x_0}^0$ we use a different embedding
	\begin{align}
		\norm {w}_{L^m(D_{-x_0}^0)} \leq C_2(\partial \Omega) \norm {w}_{W^{1,p}(D_{-x_0}^0)} ,
	\end{align}
	where we set $p:= 2 \frac{n+1}{n+2} \in (1,2)$ and  $m:= p^* = \frac{np}{n-p} = 2 \frac{n}{n-1} >2$.
 Then we get 
\begin{align}
	\int \limits_{D_{-x_0}^0} \eta(w) w^2 \psi &\leq \gamma \int \limits_{D_{-x_0}^0} w^2 \psi + \tilde{\gamma}(\gamma,m)  \int \limits_{D_{-x_0}^0} \abs {w}^m \psi \\
	&\leq \gamma \int \limits_{D_{-x_0}^0} w^2 \psi + \tilde{\gamma}(\gamma,m) \norm {\psi}_{L^\infty(D_{-x_0}^0)}  C_2^m(\partial \Omega) \norm {w}^m_{W^{1,p}(D_{-x_0}^0)}  \\
	&\leq \gamma \norm {w}^2_{L^2(D_{-x_0}^0, \psi)} + \tilde{\gamma}(\gamma,m) \norm {\psi}_{L^\infty(D_{-x_0}^0)}  C_2^m(\partial \Omega) \bra { \norm {\psi^{-\frac{1}{2}}}_{L^j(D_{-x_0}^0)}   2^\frac{2-p}{2p}  \norm {w}_{H^1(D_{-x_0}^0, \psi)}    }^m
		\end{align}
			\iftoggle{Nebenrechnungen}{%
			\\
			\textcolor{blue}{ 
				\hrulefill \\
				\emph{Nebenrechnung:}
			Sei $\frac{1}{p} = \frac{1}{2} + \frac{1}{j}$ (für $p \in (1,2)$). Dann ist
			\begin{align}
				\bra {\int (a^p g)g^{-1} + \int(b^p g)g^{-1}  }^\frac{1}{p} &\leq \bra { \norm {a^p g}_{L^\frac{2}{p}} \norm {g^{-1} }_{L^\frac{j}{p}} +\norm {b^p g}_{L^\frac{2}{p}} \norm {g^{-1} }_{L^\frac{j}{p}}      }^\frac{1}{p} \\
				&\leq \norm {g^{-1} }_{L^\frac{j}{p}}^\frac{1}{p} \bra {\norm {a^p g}_{L^\frac{2}{p}} +\norm {b^p g}_{L^\frac{2}{p}} }^\frac{1}{p} \\
				&\leq \bra { \int g^{-\frac{j}{p}}  }^\frac{1}{j} \bra {   \pra { \int a^2 g^\frac{2}{p}   }^\frac{p}{2} + \pra {  \int b^2 g^\frac{2}{p} }^\frac{p}{2}         }^\frac{1}{p} \\
				&= \norm {g^{-\frac{1}{p}}}_{L^j} \bra { \norm {a g^\frac{1}{p}}^p_{L^2} +\norm {b g^\frac{1}{p}}^p_{L^2}    }^\frac{1}{p} \\
				&\leq  \norm {g^{-\frac{1}{p}}}_{L^j}  2^\frac{2-p}{2p}     \bra { \norm {a g^\frac{1}{p}}^2_{L^2} + \norm {b g^\frac{1}{p}}^2_{L^2}    }^\frac{1}{2}
			\end{align}
			We used that (for $1=\frac{p}{2} + \frac{2-p}{2}$)
			\begin{align}
				(a^p+b^p)^\frac{1}{p} = \abs {(1,1) \cdot (a^p,b^p)}^\frac{1}{p} \leq (1^\frac{2-p}{2}+1^\frac{2-p}{2})^{\frac{2-p}{2} \frac{1}{p}} \cdot (a^2+b^2)^{\frac{p}{2} \frac{1}{p}} = 2^\frac{2-p}{2p} \sqrt{a^2+b^2}
			\end{align}
			In the end set $g:= \psi^\frac{p}{2}$.
				\hrulefill 
				\\	
			}
		}{%
		}

	where we again used Hölder's inequality with $\frac{1}{p} = \frac{1}{2} + \frac{1}{j}$ (recall that $p \in (1,2)$).
	Since $m<q$ for all $b>0$ arbitrary but fixed, it holds that for any $z \in \R$:
	\begin{align}
	\abs {z}^{m} \leq b \abs {z}^2 + b^\frac{m-q}{m-2} \abs {z}^q .
	\end{align} 
	Note that $\frac{m-q}{m-2}<0$. Using this estimate we get for arbitrary but fixed $b>0$:
	\begin{align}
		&\int \limits_{D_R^0} \eta(w) w^2 \psi \leq \gamma \norm {w}^2_{L^2(D_R^{-x_0},\psi)}  \tilde{\gamma}(\gamma,q) \psi(-x_0)^{1-\frac{q}{2}} C_1^q(\partial \Omega) \norm {w}^q_{H^1(D_R^{-x_0},\psi)}  \\
		&+ \gamma \norm {w}^2_{L^2(D_{-x_0^0},\psi)} + \tilde{\gamma}(\gamma,m) \norm {\psi}_{L^\infty(D_{-x_0}^0)}  C_2^m(\partial \Omega) \norm {\psi^{-\frac{1}{2}}}^m_{L^j(D_{-x_0}^0)}   2^\frac{(2-p)m}{2p} \\
		&\cdot \bra {b \norm {w}^2_{H^1(D_{-x_0}^0, \psi)} + b^\frac{m-q}{m-2} \norm {w}^q_{H^1(D_{-x_0}^0, \psi)}   } \\
		&\leq \bra {\gamma  + \tilde{\gamma}(\gamma,m) \norm {\psi}_{L^\infty(D_{-x_0}^0)}  C_2^m(\partial \Omega) \norm {\psi^{-\frac{1}{2}}}^m_{L^j(D_{-x_0}^0)}   2^\frac{(2-p)m}{2p} b } \norm {w}^2_{H^1(D_{R}^0, \psi)} \\
		&+ 2 \max \left \{ \tilde{\gamma}(\gamma,q) \psi(-x_0)^{1-\frac{q}{2}} C_1^q(\partial \Omega),    \right . \\
		&\quad \left .\tilde{\gamma}(\gamma,m) \norm {\psi}_{L^\infty(D_{-x_0}^0)}  C_2^m(\partial \Omega) \norm {\psi^{-\frac{1}{2}}}^m_{L^j(D_{-x_0}^0)}   2^\frac{(2-p)m}{2p} b^\frac{m-q}{m-2} \right \} \norm {w}^q_{H^1(D_{R}^0, \psi)} 
		.
	\end{align}
	In order to finish the proof we require
	\begin{enumerate}[label=\Roman*)]
		\item $\gamma \leq \frac{\alpha}{4}$
	     \item $\tilde{\gamma}(\gamma,m) \norm {\psi}_{L^\infty(D_{-x_0}^0)}  C_2^m(\partial \Omega) \norm {\psi^{-\frac{1}{2}}}^m_{L^j(D_{-x_0}^0)}   2^\frac{(2-p)m}{2p} b   \leq \frac{\alpha}{4}$
	     \item \label{crit_III} $ \max \left \{ \tilde{\gamma}(\gamma,q) \psi(-x_0)^{1-\frac{q}{2}} C_1^q(\partial \Omega), \right . \\ \left .
	     \tilde{\gamma}(\gamma,m) \norm {\psi}_{L^\infty(D_{-x_0}^0)}  C_2^m(\partial \Omega) \norm {\psi^{-\frac{1}{2}}}^m_{L^j(D_{-x_0}^0)}   2^\frac{(2-p)m}{2p} b^\frac{m-q}{m-2} \right \} \norm {w}^{q-2}_{H^1(D_{R}^0, \psi)} \\ \leq \frac{\alpha}{4} $
	\end{enumerate}
Let us now set 
\begin{align}
	\gamma &:= \frac{\alpha}{4} \\
     b &:=\frac{\alpha}{4} \bra {\tilde{\gamma}(\gamma,m) \norm {\psi}_{L^\infty(D_{-x_0}^0)}  C_2^m(\partial \Omega) \norm {\psi^{-\frac{1}{2}}}^m_{L^j(D_{-x_0}^0)}   2^\frac{(2-p)m}{2p}  }^{-1}
\end{align}
Let us note that
\begin{align}
	\psi(x_1,y) = \psi(-x_0) \exp \bra { \int \limits_{-x_0}^{x_1} -k_1(\zeta,y) \dx{\zeta}     }
\end{align}
and using that
\begin{align}
	\norm {\psi }_{L^\infty(D_{-x_0}^0)} &=  \psi(-x_0)   \norm {  \exp \bra {-  \int \limits_{-x_0}^{x_1} k_1(\zeta,y) \dx{\zeta} }  }_{L^\infty(D_{-x_0}^0)} \\
	\norm {\psi^{-\frac{1}{2}} }_{L^j(D_{-x_0}^0)} &= \psi^{-\frac{1}{2}}(-x_0)    \pra {\int \limits_{D_{-x_0}^0}     \exp \bra {\frac{j}{2}   \int \limits_{-x_0}^{x_1} k_1(\zeta,y) \dx{\zeta} } \dx{x_1} \dx{y}   }^\frac{1}{j}
\end{align}
Putting this into \ref{crit_III} we get
\begin{align}
	\norm {w}_{H^1(D_R^0,\psi)} &\leq \bra {\frac{\alpha}{2}}^\frac{1}{q-2} \psi^\frac{1}{2}(-x_0) \max \left \{   \tilde{\gamma} \bra { \frac{\alpha}{4},q   } C_1^q(\partial \Omega)  , \right . \\
	& \left . \left ( \tilde{\gamma} \bra {\frac{\alpha}{2} , m} C_2^m(\partial \Omega)  2^\frac{(2-q)m}{2p} \norm {  \exp \bra {-  \int \limits_{-x_0}^{x_1} k_1(\zeta,y) \dx{\zeta} }   }_{L^\infty(D_{-x_0}^0)}   \right . \right . \\
	& \left . \left .\pra {\int \limits_{D_{-x_0}^0}     \exp \bra {\frac{j}{2}   \int \limits_{-x_0}^{x_1} k_1(\zeta,y) \dx{\zeta} } \dx{x_1} \dx{y}   }^\frac{m}{j}  \right )^\frac{q-2}{m-2}  \right \}^{-\frac{1}{q-2}} := \delta.
\end{align}

\end{proof}

With Lemma \ref{lemma:auxiliary_lemma_fixed_R} being proved we are in the position to prove Proposition \ref{proposition:blocking:fixed_R}.

\begin{proof}[Proposition \ref{proposition:blocking:fixed_R}]
The strategy of this proof is to show that for all $w \in H^1_{0,1}(D_R^a, \psi )$ such that $\norm {w-w_0}_{H^1(D_R^a, \psi )} = \delta$ it holds that 
\begin{align}
	J_{D_R^a}(w) > J_{D_R^a}(w_0) .
\end{align}
Since $J_{D_R^a}$ is weakly lower semicontinuous, bounded below and coercive, $J_{D_R^a}$ has a local minimizer among all the functions $w \in H^1_{0,1}(D_R^a, \psi )$ such that $\norm {w-w_0}_{H^1(D_R^a,\psi )} \leq \delta$. 
And since we have a local minimizer that does not lie on the boundary of $\norm {w-w_0}_{H^1(D_R^a,\psi )} = \delta$ we derive an Euler-Lagrange-equation. 
Let us first note that from $\norm {w-w_0}_{H^1(D_R^a,\psi )} \leq \delta$ it follows that $\norm {w-w_0}_{H^1(D_R^0,\psi )} \leq \delta$. 

And to make use of Lemma \ref{lemma:auxiliary_lemma_fixed_R} we split the functional as follows. For any $w \in H^1_{0,1}(D_R^a,\psi )$ it holds that 
\begin{align}
	J_{D_R^a}(w) -J_{D_R^a}(w_0) = \underbrace {	J_{D_R^0}(w) -J_{D_R^0}(w_0) }_{:= I} + \underbrace {	J_{D_0^a}(w) -J_{D_0^a}(w_0)  }_{:=II}.
\end{align}
From Lemma \ref{lemma:auxiliary_lemma_fixed_R} it follows for the first term I 
\begin{align}
 J_{D_R^0}(w) -J_{D_R^0}(w_0) \geq \alpha \norm {w}^2_{H^1(D_R^0,\psi )} .
\end{align}
For the second term II we use the observation that there is $K>0$ such that for all $s \in \R$
\begin{align} \label{estimate:F}
F(s) \geq K(s-1)^2 .
\end{align}
With this observation we conclude that for any measurable set $A \subset D_R^a$ it holds that 
\begin{align}
	J_A(w) \geq \nu \norm {w-1}^2_{H^1(A,\psi )} ,
\end{align}
where $\nu := \min \set {K, \frac{1}{2}}$. Furthermore we estimate 
\begin{align}
	J_{D_0^a}(w_0) = \int \limits_{D_0^a} \bra {  \frac{1}{2} \bra {\frac{1}{a}}^2 +F(w_0)    } \psi \leq \underbrace { \bra { \frac{1}{2} \frac{1}{a} + a \max \limits_{s \in [0,1]} F(s)    } \abs {\Omega} }_{:= \beta} \psi(0)
\end{align}
Together with \eqref{estimate:F} we get 
\begin{align}
	J_{D_0^a}(w) - J_{D_0^a}(w_0) \geq \nu \norm {w-1}_{H^1(D_0^a,\psi )} -\beta \psi(0) .
\end{align}
In order to use the assumption that $\norm {w-w_0}_{H^1(D_R^a, \psi )} = \delta$ we estimate that
\begin{align}
	\nu \norm {w-1}^2_{H^1(D_0^a, \psi )} \geq \frac{\nu}{2} \norm {w-w_0}^2_{H^1(D_0^a,\psi )} - \nu \norm {w_0 -1}^2_{(D_0^a,\psi )}
\end{align}
using Young's inequality. Furthermore by a direct calculation
\begin{align}
	\norm {w_0-1}^2_{H^1(D_0^a,\psi )} = \int \limits_{D_0^a} \bra {  \bra { \frac{1}{a} }^2 + \bra {\frac{x_1}{a} -1    }^2     } \psi = \psi(0)  \underbrace {  \bra { \frac{1}{a} + \frac{1}{3} a   } \abs {\Omega} }_{=: \gamma} .
\end{align}
Putting these estimations together we get
\begin{align}
	J_{D_R^a}(w) -J_{D_R^a}(w_0) &\geq \frac{\nu}{2} \norm {w-w_0}^2_{H^1(D_0^a,\psi )} -\nu \psi(0) - \beta \psi(0) + \alpha \norm {w-w_0}^2_{H^1(D_R^0, \psi)} \\
	&\geq \min \set {\frac{\nu}{2} , \alpha}  \underbrace{  \norm {w-w_0}^2_{H^1(D_R^a,\psi )} }_{=\delta^2}- \nu \gamma \psi(0) - \beta \psi(0) .
\end{align}
So the proposition is proved if this is positive. This is the case if 
\begin{align}
	\eta \delta^2 - \bra {\nu \gamma + \beta} \psi(0) >0 .
\end{align}
Exploiting that 
\begin{align}
	\frac{\psi(0)}{\psi(-x_0)} = \exp \bra { - \int \limits_{-x_0}^{0} \fint \limits_\Omega k_1(x_1,y) \dx{x_1} \dx{y}       }
\end{align}
we arrive at the condition
\begin{align}
	&\frac{\eta}{\nu \gamma + \beta}  \bra { \frac{\alpha}{4} }^\frac{2}{q-2} >\exp \bra { - \int \limits_{-x_0}^{0} \fint \limits_\Omega k_1(x_1,y) \dx{x_1} \dx{y}       }
	\\ &\max \left \{ C_1^q(\partial \Omega) \tilde{\gamma} \bra { \frac{\alpha}{4},q } , \left [ \tilde{\gamma} \bra { \frac{\alpha}{4} ,m  }  C_2(\partial \Omega) 2^{\frac{(2-q)m}{2p}} \right .  \right . \\  
	&\left .  \left .    \norm { \exp \bra {- \int \limits_{-x_0}^{x_1} k_1(\zeta,y) \dx{\zeta}    }}_{L^\infty(D_{-x_0}^0)}  \bra {  \int \limits_{D_{-x_0}^0} \exp \bra { \frac{j}{2} \int \limits_{-x_0}^{x_1} k_1(\zeta,y) \dx{\zeta}      } \dx{x_1} \dx{y}    }^\frac{m}{j}           \right ]^\frac{q-2}{m-2}  \right    \}^{\frac{2}{q-2}} 
\end{align}
\end{proof}


From Proposition \ref{proposition:blocking:fixed_R} we get for any $R<-x_0-1$ existence of a local minimizer $w_R \in H^1_{0,1}(D_R^a, \psi )$ of the functional $J_{D_R^a}$ such that $\norm {w_R-w_0}_{H^1(D_R^a, \psi )} \leq \delta$. From this it follows that $w_R$ is a weak solution of 
\begin{align}
\begin{cases}
-\Delta w_R + k \cdot \nabla w_R = f(w_R) &\text{ in } D_R^a, \\
w_R(a,y) =1 &\text{ for almost all } y \in \Omega,\\
w_R(R,y) =0 &\text{ for almost all } y \in \Omega, \\
\frac{\partial w_R}{\partial \nu} = 0 &\text{ on } \partial D \cap D_R^a.
\end{cases} 
\end{align} 
Using the maximum principle we conclude that for all $R<-x_0-1: 0 \leq w_R \leq 1$ in $D_R^a$.
From this we construct a supersolution to 
\begin{align}
\begin{cases}
\partial_t u -\Delta u +k \cdot \nabla u= f(u) \text{ for all } (t,x) \in \R\times D \\
u(t,x) - \phi(x+ct) \rightarrow 0 \text{ as } t \rightarrow -\infty \text{ uniformly in } D
\end{cases} 
\end{align} 
by passing to the limit $R \rightarrow -\infty$ and extending by $1$ onto $D_a^\infty$.

\begin{prop}[Existence of a stationary supersolution] \label{proposition:existence_stationary_supersolution}
	Assume that condition \eqref{blocking_condition} holds and let $w_R$ be the local minimizer of the energy functional $J_{D_R^a}$ as in Proposition \ref{proposition:blocking:fixed_R}, then $(w_R)_{R<-x_0-1}$ converges up to a subsequence in $C^2_\text{loc}(D_{-\infty}^a)$ to a solution $w_\infty$ of
	\begin{align} \label{w_infty_equation}
	\begin{cases}
	- \Delta w_\infty+k \cdot \nabla w_\infty  = f(w_\infty) &\text{ in } D_{-\infty}^a, \\
	w(a,y)=1 &\text{ for almost all } y \in \Omega,
	\end{cases}
	\end{align}
	such that $w_\infty(x) \rightarrow 0$ as $x_1 \rightarrow -\infty$.
\end{prop}

\begin{proof}
	As $0\leq w_R \leq 1$ for all $R>-x_0-1$ and using Schauder estimates there exists a subsequence $(R_n)_{n \in \N}$ with $R_n \searrow -\infty$ as $n \rightarrow \infty$ such that $w_{R_n} \rightarrow w_\infty$ in $C^2_\text{loc}(D_{-\infty}^a)$ as $n \rightarrow \infty$.
	It remains to prove that the limit $w_\infty$ satisfies $w_\infty \rightarrow 0$ as $x_1 \rightarrow -\infty$. By Fatou's Lemma we find
	\begin{align} \label{Fatou_w_infty_w_0_bound}
	\norm {w_0-w_\infty}^2_{H^1(D_{-\infty}^a, \psi )} \leq \liminf \limits_{n \rightarrow \infty} \norm {(w_0-w_{R_n})  \chi_{\set {R_n<x_1<a}}   }^2_{H^1(D_{-\infty}^a, \psi )} \leq \delta^2.
	\end{align}
	Then arguing by contradiction, we assume that there exists $\eta >0$ and a sequence $(x_n)_{n \in \N}$, such that $(x_n)_1 \rightarrow -\infty$ as $n \rightarrow \infty$ and $w_\infty(x_n) >\eta$ for all $n \in \N$. Since $w_\infty \in C^2_\text{loc}$ and $w_\infty$ is a bounded solution of \eqref{w_infty_equation}, by standard parabolic estimates we know that $\abs { \nabla w_\infty} \leq C$ for some constant $C>0$. It follows that for all $x \in B_\frac{\eta}{2C}(x_n) $
	\begin{align}
	\abs {w_\infty(x) - w_\infty(x_n)} \leq \max \limits_{x \in B_\frac{\eta}{2C}(x_n) } \abs {\nabla w_\infty} \abs {x-x_n}
	\end{align}
	Hence $w_\infty(x) \geq \frac{\eta}{2}$ for all $x \in B_\frac{\eta}{2C}(x_n)$ and all $n \in \N$. This yields that
	\begin{align}
	\norm {w_\infty -w_0}^2_{L^2(D_{-\infty}^a, \psi )} \geq \psi(-x_0) \sum \limits_{I}^\infty \bra {\frac{\eta}{2} }^2 \frac{\eta}{C} = \infty,
	\end{align}
	where $I:= \set {i \in \N : \abs {x_i -x_j}  > \frac{\eta}{2C} \text{ for all } \N \setminus \set {i} \text{ and } (x_i)_1 <-x_0   }$ and obviously $\#I = \infty$. But this is a contradiction to \eqref{Fatou_w_infty_w_0_bound} and thereby the Proposition is proved.
\end{proof}

The proof of Theorem \ref{theorem:blocking} is now nothing but applying Proposition \ref{proposition:existence_stationary_supersolution} and a comparison principle.

\begin{proof}[Proof of Theorem \ref{theorem:blocking}]

	Let us now take $w_\infty$ as in Proposition \ref{proposition:existence_stationary_supersolution} and let us  extend $w_\infty$ by $1$ onto all of $\R$. We set
	\begin{align}
	\tilde{w}_\infty(x) := 
	\begin{cases}
	w_\infty(x) &, \text{ if } x_1 \leq a \\
	1 &, \text{ else}.
	\end{cases}
	\end{align}
	Thus $\tilde{w}_\infty(x) $ is a supersolution of the parabolic problem 
	\begin{align}
	\partial_t u - \Delta u + k(x)  \cdot \nabla  u  = f(u) &\text{ for } (t,x) \in \R \times D ,\\
	\frac{\partial u}{\partial \nu} = 0 &\text{ on } \R \times \partial D.
	\end{align}
	Furthermore it holds that
	\begin{align}
	\lim \limits_{t \rightarrow -\infty} \inf \limits_{x \in D}  \bra {  \tilde{w}_\infty(x) - \phi(x_1+ct) } \geq 0.
	\end{align}
	Indeed 
	\begin{align}
	&\text{ for } x_1 \geq a, \text{ for all } t \in \R:  \tilde{w}_\infty(x) - \phi(x_1+ct) = 1- \phi(x_1+ct) \geq 0 \\
	&\text{ for } x_1 < a, \text{ for all } t \in \R:  \tilde{w}_\infty(x) - \phi(x_1+ct) \geq \tilde{w}_\infty(x) - \phi(a+ct) \rightarrow \tilde{w}_\infty(x) \geq 0 \\
	&\qquad \text{ uniformly in } D_{-\infty}^a \text{ as } t \rightarrow -\infty.
	\end{align}
	Using the generalized maximum principle (Lemma 3.2 in \cite{front_blocking}), we conclude that
	\begin{align}
	u(t,x) \leq \tilde{w}_\infty(x) \text{ for all } t \in \R \text{ and } x \in D.
	\end{align} 
	Hence the stationary supersolution $\tilde{w}_\infty$ blocks the invasion of the stationary state $1$ into the right.
	
\end{proof}



\nomenclature{$\mathcal{L}^N$}{The $N$-dimensional Lebesgue measure}%
\nomenclature{$B_R(x)$}{The open ball of radius $R$ centred in $x$}%
\nomenclature{$C_R(x)$}{The open cube of side length $R$ centred in $x$}%
\nomenclature{$\omega_N$}{The $N$-dimensional Lebesgue-measure of the $N$-dimensional unit ball}%

%
%

\bibliographystyle{abbrv}
\bibliography{Change_of_speed-3.bib}

\begin{thebibliography}{10}

\bibitem{front_blocking}
H.~Berestycki, J.~Bouhours, and G.~Chapuisat.
\newblock Front blocking and propagation in cylinders with varying cross
  section.
\newblock {\em Calc. Var. Partial Differential Equations}, 55(3):Paper No. 44,
  32, 2016.

\bibitem{BerestyckiHamelPeriodicExcitableMedia}
H.~Berestycki and F.~Hamel.
\newblock Front propagation in periodic excitable media.
\newblock {\em Comm. Pure Appl. Math.}, 55(8):949--1032, 2002.

\bibitem{BerestyckiHamelGeneralizedTransitionFronts}
H.~Berestycki and F.~Hamel.
\newblock Generalized transition waves and their properties.
\newblock {\em Comm. Pure Appl. Math.}, 65(5):592--648, 2012.

\bibitem{MatanoObst}
H.~Berestycki, F.~Hamel, and H.~Matano.
\newblock Bistable traveling waves around an obstacle.
\newblock {\em Comm. Pure Appl. Math.}, 62(6):729--788, 2009.

\bibitem{traveling_fronts_in_cylinders}
H.~Berestycki and L.~Nirenberg.
\newblock Travelling fronts in cylinders.
\newblock {\em Ann. Inst. H. Poincar\'e Anal. Non Lin\'eaire}, 9(5):497--572,
  1992.

\bibitem{suden_opening}
G.~Chapuisat and E.~Grenier.
\newblock Existence and nonexistence of traveling wave solutions for a bistable
  reaction-diffusion equation in an infinite cylinder whose diameter is
  suddenly increased.
\newblock {\em Comm. Partial Differential Equations}, 30(10-12):1805--1816,
  2005.

\bibitem{front_blocking_vs_propagation_drift}
S.~{Eberle}.
\newblock Front blocking versus propagation in the presence of drift
  disturbance in the direction of propagation.
\newblock {\em arXiv}, arXiv:1803.03102, 2018.

\bibitem{change_of_speed_1}
S.~Eberle.
\newblock A heteroclinic orbit connecting traveling waves pertaining to
  different nonlinearities.
\newblock {\em Journal of Differential Equations, in press},
  https://doi.org/10.1016/j.jde.2018.03.007, 2018.

\bibitem{change_of_speed2}
S.~Eberle.
\newblock A heteroclinic orbit connecting traveling waves pertaining to
  different nonlinearities in a channel with decreasing cross section.
\newblock {\em Nonlinear Analysis}, 172:99--114, 7 2018.

\bibitem{FifeMcLeod}
P.~C. Fife and J.~B. McLeod.
\newblock The approach of solutions of nonlinear diffusion equations to
  travelling front solutions.
\newblock {\em Arch. Ration. Mech. Anal.}, 65(4):335--361, 1977.

\bibitem{Zlatos}
A.~Zlato{\v{s}}.
\newblock Generalized traveling waves in disordered media: existence,
  uniqueness, and stability.
\newblock {\em Arch. Ration. Mech. Anal.}, 208(2):447--480, 2013.

\end{thebibliography}

\end{document}